\documentclass[12pt]{article}
\usepackage{amssymb}
\usepackage{amsthm}
\usepackage{amsmath}
\usepackage{graphicx}
\usepackage{enumerate}
\usepackage{pict2e}
\usepackage{framed}

\DeclareMathOperator{\Cl}{Cl}
\DeclareMathOperator{\BCl}{\mathbf{Cl}}

\DeclareMathOperator{\At}{At}

\newtheorem{theorem}{Theorem}[section]

\newtheorem{lemma}[theorem]{Lemma}

\newtheorem{remark}[theorem]{Remark}
\newtheorem{example}[theorem]{Example}
\newtheorem{corollary}[theorem]{Corollary}
\title{Induced orthogonality in semilattices with $0$ and in pseudocomplemented lattices and posets}
\author{Ivan~Chajda, Miroslav~Kola\v r\'ik and Helmut~L\"anger$^2$}
\date{}

\begin{document}
	
\footnotetext[1]{Support of the research of the first author by the Czech Science Foundation (GA\v CR), project 24-14386L, entitled ``Representation of algebraic semantics for substructural logics'', and by IGA, project P\v rF~2024~011, is gratefully acknowledged.}

\footnotetext[2]{Corresponding author}
	
\maketitle
	
\begin{abstract}
  On an arbitrary meet-semilattice $\mathbf S=(S,\wedge,0)$ with $0$ we define an orthogonality relation and investigate the lattice $\BCl(\mathbf S)$ of all subsets of $S$ closed under this orthogonality. We show that if $\mathbf S$ is atomic then $\BCl(\mathbf S)$ is a complete atomic Boolean algebra. If $\mathbf S$ is a pseudocomplemented lattice, this orthogonality relation can be defined by means of the pseudocomplementation. Finally, we show that if $\mathbf S$ is a complete pseudocomplemented lattice then $\BCl(\mathbf S)$ is a complete Boolean algebra. For pseudocomplemented posets a similar result holds if the subset of pseudocomplements forms a complete lattice satisfying a certain compatibility condition.
\end{abstract}
	
{\bf AMS Subject Classification:} 06B05, 06C15, 06D15, 06D25, 08A11
	
{\bf Keywords:} meet-semilattice, lattice, orthogonality, closed subset, ortholattice, Boolean algebra, pseudocomplemented lattice, pseudocomplemented poset
	
\section{Introduction}

By an orthogonality on a set $S$ is meant a symmetric and irreflexive binary relation~$\perp$. For every subset $A$ of $S$ put $A^\perp:=\{x\in S\mid x\perp y\text{ for all }y\in A\}$. A subset $A$ of $S$ is called {\em closed} if $A^{\perp\perp}=A$. It is well-known that the lattice of all closed subsets of a given set $S$ forms a complete ortholattice. Several authors studied when this lattice is orthomodular, see \cite D, \cite J and \cite L and, for Hilbert spaces, see also \cite{LV} and \cite{PV}.

We slightly modify the definition of an orthogonality relation on a set $S$ with a distinguished element $0$. We require such a relation $\perp$ to be symmetric and to have the property that $0$ is the only element $x$ of $S$ satisfying $x\perp x$. We study lattices of closed subsets of meet-semilattices or lattices with $0$ where the orthogonality relation is derived by means of the meet-operation. In the case of pseudocomplemented lattices, the orthogonality relation can be introduced via the pseudocomplementation. It is expected that if a meet-semilattice $\mathbf S$ with $0$ is considered instead of a set and the orthogonality is derived from the meet-operation then the lattice $\BCl(\mathbf S)$ of closed subsets will satisfy a stronger property then orthomodularity. We show that if $\mathbf S$ is atomic then $\BCl(\mathbf S)$ forms a complete atomic Boolean algebra. Especially, we consider the case when $\mathbf S$ is either the modular lattice $\mathbf M_n$ or a Boolean algebra. If $\mathbf L$ is a complete pseudocomplemented lattice then $\BCl(\mathbf L)$ turns out to be a complete Boolean algebra. Several examples illuminating these concepts and results are included in Sections~2 and 3. If $\mathbf P=(P,\le)$ is a pseudocomplemented poset then we can define orthogonality on $\mathbf P$ by means of pseudocomplementation in the same way as it was done in the case of pseudocomplemented lattices. It turns out that if its subset $P^*$ of all pseudocomplements is a complete lattice satisfying a certain compatibility condition then the complete ortholattice $\BCl(\mathbf P)$ of closed subsets of $\mathbf P$ is again a Boolean algebra. We also characterize finite pseudocomplemented posets whose poset of pseudocomplements forms a meet-semilattice by means of a forbidden configuration.

\section{Meet-semilattices and lattices with $0$}

The aim of this section is to show how orthogonality can be introduced in a meet-semilattice with $0$ and to describe the structure of the lattice of closed subsets.

In the following let $\mathbf S=(S,\wedge,0)$ be a meet-semilattice with $0$, $a\in S$ and $A\subseteq S$. On $S$ we define a binary relation $\perp$, the so-called {\em orthogonality}, as follows:
\[
x\perp y\text{ if and only if }x\wedge y=0
\]
($x,y\in S$). Clearly $\perp$ is symmetric, and if $x\perp x$ then $x=0$. Thus, if $\mathbf S$ is non-trivial then $\perp$ is not reflexive. Further define
\[
	A^\perp:=\{x\in S\mid x\perp y\text{ for all }y\in A\}.
\]
For the sake of brevity, instead of $\{a\}^\perp$ we simply write $a^\perp$. Now we are ready to define our main concept.

The set $A$ is called {\em closed} if $A^{\perp\perp}=A$. Let $\Cl(\mathbf S)$ denote the set of all closed subsets of $S$. Obviously, $A$ is closed if and only if there exists some subset $B$ of $S$ with $B^\perp=A$.

\begin{remark}\label{rem1}
	The pair $(^\perp,^\perp)$ is the Galois connection between $(2^S,\subseteq)$ and $(2^S,\subseteq)$ induced by the relation $\perp$. This implies that $\BCl(\mathbf S):=\big(\Cl(\mathbf S),\subseteq\big)$ is a complete lattice with smallest element $S^\perp=\{0\}$ and greatest element $\emptyset^\perp=S$ and
	\begin{align*}
		     \bigvee_{i\in I}A_i & =\left(\bigcup_{i\in I}A_i\right)^{\perp\perp}, \\
   		   \bigwedge_{i\in I}A_i & =\bigcap_{i\in I}A_i
	\end{align*}
	for all families $A_i,i\in I,$ of elements of $\Cl(\mathbf S)$. Moreover, we have
\[
\left(\bigcup_{i\in I}A_i\right)^\perp=\bigcap_{i\in I}A_i^\perp
\]
for all families $A_i,i\in I,$ of subsets of $S$, especially
\[
A^\perp=\bigcap_{x\in A}x^\perp
\]
for all $A\subseteq S$. From the last equality we deduce that $\Cl(\mathbf S)$ is the smallest subset of $2^S$ including $\{x^\perp\mid x\in S\}$ and being closed under arbitrary intersections. We have that $^\perp$ is antitone and $A\subseteq A^{\perp\perp}$, $A^\perp=A^{\perp\perp\perp}$ and $A\cap A^\perp=\{0\}$ for all $A\subseteq S$.
\end{remark}

The following lemma is well-known (cf.\ \cite J).

\begin{lemma}\label{lem3}
	If $\mathbf S=(S,\wedge,0)$ is a meet-semilattice with $0$ then
	\[
	\BCl(\mathbf S):=\big(\Cl(\mathbf S),\vee,\cap,{}^\perp,\{0\},S\big)
	\]
is a complete ortholattice with
	\begin{align*}
	  \bigvee_{i\in I}A_i & =\left(\bigcup_{i\in I}A_i\right)^{\perp\perp}, \\
	\bigwedge_{i\in I}A_i & =\bigcap_{i\in I}A_i
\end{align*}
for all families $A_i,i\in I,$ of elements of $\Cl(\mathbf S)$.
\end{lemma}

For a poset $\mathbf P=(P,\le)$ and $a\in P$ define
\begin{align*}
	(a] & :=\{x\in P\mid x\le a\}, \\
	[a) & :=\{x\in P\mid x\ge a\}.
\end{align*}
Especially, in the set of all subsets of a set $S$, if $B\subseteq A$ then $(B]=\{C\mid C\subseteq B\}$.

If $\mathbf L=(L,\vee,\wedge,0)$ is a lattice with $0$ then notions like orthogonality and closed subset have the same meaning as for the corresponding meet-semilattice $(L,\wedge,0)$ with $0$.

For the reader's convenience, we illustrate the aforementioned concepts by the following example.

\begin{example}\label{ex1}
	Consider the lattice $\mathbf L_1=(L,\vee,\wedge,0)$ with $0$ depicted in Fig.~1:
	\begin{center}
		\setlength{\unitlength}{1.3mm}
		\begin{picture}(110,60)
			\put(10,25){\circle*{1.5}}
			\put(25,10){\circle*{1.5}}
			\put(25,40){\circle*{1.5}}
			\put(40,25){\circle*{1.5}}
			\put(70,40){\circle*{1.5}}
			\put(85,25){\circle*{1.5}}
			\put(85,40){\circle*{1.5}}
			\put(85,55){\circle*{1.5}}
			\put(100,40){\circle*{1.5}}
			\put(25,10){\line(-1,1){15}}
			\put(25,10){\line(1,1){15}}
			\put(25,40){\line(-1,-1){15}}
			\put(25,40){\line(1,-1){15}}
			\put(85,25){\line(0,1){30}}
			\put(85,25){\line(-1,1){15}}
			\put(85,25){\line(1,1){15}}
			\put(85,55){\line(-1,-1){15}}
			\put(85,55){\line(1,-1){15}}
			\put(10,25){\line(4,1){60}}
			\put(25,10){\line(4,1){60}}
			\put(25,40){\line(4,1){60}}
			\put(40,25){\line(4,1){60}}
			\put(24.1,6){$0$}
			\put(83.7,57){$1$}
			\put(7,24){$a$}
			\put(40.5,22){$b$}
			\put(24,42){$c$}
			\put(84,21.5){$d$}
			\put(101.5,39){$f$}
			\put(86.5,39){$g$}
			\put(67.7,41){$e$}
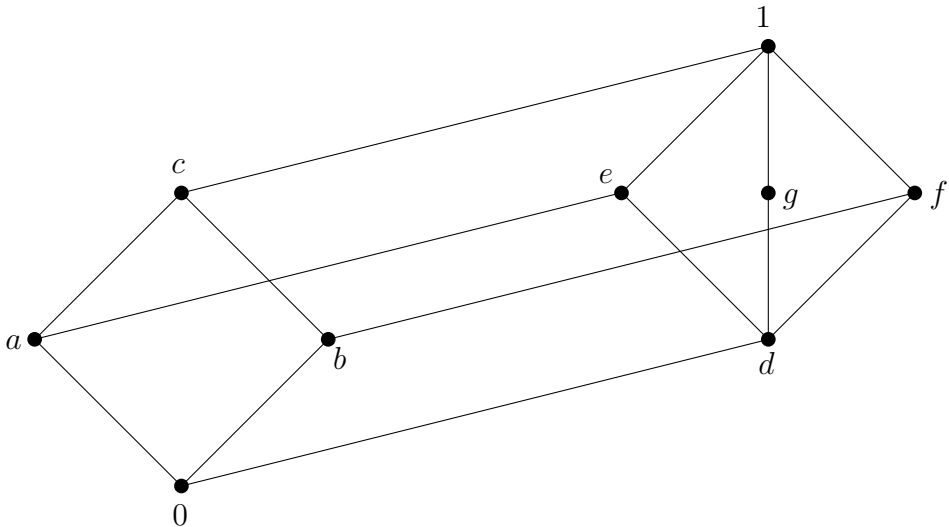
			\put(38,0){{\rm Fig.~1: Lattice $\mathbf L_1$ with $0$}}
		\end{picture}
	\end{center}
	We have $0^\perp=L$, $a^\perp=\{0,b,d,f,g\}$, $b^\perp=\{0,a,d,e,g\}$, $c^\perp=(g]$, $d^\perp=g^\perp=(c]$, $e^\perp=(b]$, $f^\perp=(a]$ and $1^\perp=\{0\}$. Since $\{x^\perp\mid x\in L\}$ is closed with respect to intersections we obtain $\Cl(\mathbf L_1)=\{x^\perp\mid x\in L\}$. The ortholattice $\BCl(\mathbf L_1)$ is visualized in Fig.~2:
	\begin{center}
		\setlength{\unitlength}{1.3mm}
		\begin{picture}(50,60)
			\put(10,25){\circle*{1.5}}
			\put(10,40){\circle*{1.5}}
			\put(25,10){\circle*{1.5}}
			\put(25,25){\circle*{1.5}}
			\put(25,40){\circle*{1.5}}
			\put(25,55){\circle*{1.5}}
			\put(40,25){\circle*{1.5}}
			\put(40,40){\circle*{1.5}}
			\put(10,25){\line(0,1){15}}
			\put(40,25){\line(0,1){15}}
			\put(25,10){\line(0,1){15}}
			\put(25,40){\line(0,1){15}}
			\put(25,10){\line(-1,1){15}}
			\put(25,10){\line(1,1){15}}
			\put(25,25){\line(-1,1){15}}
			\put(25,25){\line(1,1){15}}
			\put(25,40){\line(-1,-1){15}}
			\put(25,40){\line(1,-1){15}}
			\put(25,55){\line(-1,-1){15}}
			\put(25,55){\line(1,-1){15}}
			\put(22.7,6){$\{0\}$}
			\put(4,24){$(a]$}
			\put(27,24){$(g]$}
			\put(42,24){$(b]$}
			\put(42,39){$\{0,b,d,f,g\}$}
			\put(-9,39){$\{0,a,d,e,g\}$}
			\put(27,39){$(c]$}
			\put(24.2,57){$L$}
			\put(6,0){{\rm Fig.~2: Ortholattice $\BCl(\mathbf L_1)$}}
		\end{picture}
	\end{center}
\end{example}

In the sequel for any set $A$ let $\mathbf2^A$ denote the Boolean algebra $(2^A,\subseteq)$ of all subsets of $A$. If $A$ is finite, say $|A|=n$, then we write $\mathbf2^n$ instead of $\mathbf2^A$. Moreover, for every positive integer $n$ let $\mathbf M_n$ denote the lattice consisting of an $n$-element antichain together with $0$ and $1$, see Fig.~3.

\begin{center}
  \setlength{\unitlength}{1.3mm}
  \begin{picture}(80,45)
    \put(45,25){\circle*{0.5}}
    \put(47.5,25){\circle*{0.5}}
    \put(50,25){\circle*{0.5}}
    \put(10,25){\circle*{1.5}}
    \put(25,25){\circle*{1.5}}
    \put(40,10){\circle*{1.5}}
    \put(40,40){\circle*{1.5}}
    \put(70,25){\circle*{1.5}}
    \put(40,10){\line(-2,1){30}}
    \put(40,10){\line(2,1){30}}
    \put(40,10){\line(-1,1){15}}
    \put(40,40){\line(-2,-1){30}}
    \put(40,40){\line(2,-1){30}}
    \put(40,40){\line(-1,-1){15}}
    \put(39,6){$0$}
    \put(39,42){$1$}
    \put(72,24){$a_n$}
    \put(5.8,24){$a_1$}
    \put(27,24){$a_2$}
    \put(27,0){{\rm Fig.~3: Lattice $\mathbf M_n$}}
  \end{picture}
\end{center}

The result illustrated in Example~\ref{ex1} can be extended to arbitrary atomic semilattices with $0$.

\begin{theorem}\label{th1}
	Let $\mathbf S=(S,\wedge,0)$ be an atomic meet-semilattice with $0$ and let $\At\mathbf S$ denote the set of all atoms of $\mathbf S$. Then $\BCl(\mathbf S)\cong\mathbf2^{\At\mathbf S}$.
\end{theorem}

\begin{proof}
	Let $A\subseteq S$ and put
	\[
	B:=\{x\in S\mid x\not\ge y\text{ for all }y\in\At\mathbf S\text{ with }[y)\cap A\ne\emptyset\}.
	\]
	We show $A^\perp=B$. Let $a\in A^\perp$. Suppose $a\notin B$. Then there exists some $b\in\At\mathbf S$ with $a\ge b$ and $[b)\cap A\ne\emptyset$. Let $c\in [b)\cap A$. Then $c\in A$ and $0<b\le a\wedge c$ contradicting $a\in A^\perp$. Hence $a\in B$. Conversely, let $a\in B$. Suppose $a\notin A^\perp$. Then there exists some $d\in A$ with $a\wedge d\ne0$. Since $\mathbf S$ is atomic there exists some $e\in\At\mathbf S$ with $e\le a\wedge d$. But then $e\in\At\mathbf S$ and $[e)\cap A\supseteq\{d\}\ne\emptyset$, but $a\ge e$ contradicting $a\in B$. Hence $a\in A$. This completes the proof of $A^\perp=B$. For every $C\subseteq\At\mathbf S$ put
	\[
	f(C):=\{x\in S\mid x\not\ge y\text{ for all }y\in(\At\mathbf S)\setminus C\}.
	\]
	Then $\Cl(\mathbf S)=\{f(C)\mid C\subseteq\At\mathbf S\}$. Now let $C,D\subseteq\At\mathbf S$. If $C\subseteq D$ then $f(C)\subseteq f(D)$. If, conversely, $f(C)\subseteq f(D)$ then
	\[
	C=(\At\mathbf S)\cap f(C)\subseteq(\At\mathbf S)\cap f(D)=D.
	\]
	This shows that $f$ is an isomorphism from the Boolean algebra $\mathbf2^{\At\mathbf S}$ to $\BCl(\mathbf S)$ completing the proof of the theorem.
\end{proof}

In particular cases of the lattice $\mathbf L$, we can precise the lattice $\BCl(\mathbf L)$ as follows.

\begin{corollary}
	\
	\begin{enumerate}[{\rm(i)}]
		\item If $\mathbf L=(L,\vee,\wedge)$ denotes the modular lattice consisting of a non-empty antichain $A$ together with $0$ and $1$ then $\BCl(\mathbf L)\cong\mathbf2^A$.
		\item We have $\BCl(\mathbf M_n)\cong\mathbf2^n$ for every positive integer $n$.
		\item We have $\BCl(\mathbf2^A)\cong\mathbf2^A$ for every set $A$.
		\item We have $\BCl(\mathbf L)\cong\mathbf L$ for every complete atomic Boolean algebra.
	\end{enumerate}
\end{corollary}

The last statement follows from the fact that every complete atomic Boolean algebra is isomorphic to $\mathbf2^A$ for some suitable set $A$.

The next theorem enables us to determine the structure of a number of lattices of closed subsets for various direct products of modular and/or Boolean lattices.

\begin{theorem}
	For each $i\in I$ let $\mathbf S_i=(S_i,\wedge,0)$ be a meet-semilattice with $0$. Then
\[
\BCl\left(\prod_{i\in I}\mathbf S_i\right)=\prod_{i\in I}\BCl(\mathbf S_i).
	\]
\end{theorem}

\begin{proof}
	For $j\in I$ let $p_j$ denote the $j$-th projection from $\prod\limits_{i\in I}S_i$ onto $S_j$. Then for $A\subseteq\prod\limits_{i\in I}S_i$ we have
	\begin{align*}
	A^\perp & =\{x\in\prod_{i\in I}S_i\mid x\wedge y=0\text{ for all }y\in A\}= \\
	        & =\{x\in\prod_{i\in I}S_i\mid p_j(x)\wedge p_j(y)=0\text{ for all }j\in I\text{ and all }y\in A\}= \\
	        & =\{x\in\prod_{i\in I}S_i\mid p_j(x)\wedge y_j=0\text{ for all }j\in I\text{ and all }y_j\in p_j(A)\}=\prod_{i\in I}\big(p_i(A)\big)^\perp.
	\end{align*}
Hence $\Cl\left(\prod\limits_{i\in I}\mathbf S_i\right)=\prod\limits_{i\in I}\Cl(\mathbf S_i)$. If $A_i,B_i\subseteq S_i$ for all $i\in I$ then $\prod\limits_{i\in I}A_i\subseteq\prod\limits_{i\in I}B_i$ if and only if $A_i\subseteq B_i$ for all $i\in I$. Thus, finally we obtain $\BCl\left(\prod\limits_{i\in I}\mathbf S_i\right)=\prod\limits_{i\in I}\BCl(\mathbf S_i)$.
\end{proof}

\begin{corollary}
	$\BCl(\mathbf M_n\times\mathbf2^m)\cong\mathbf2^n\times\mathbf2^m\cong\mathbf2^{n+m}$.
\end{corollary}

\section{Pseudocomplemented lattices}

We can consider closed subsets also in pseudocomplemented posets and lattices. Recall that a (bounded) {\em poset} $(P,\le,0,1)$ is called {\em pseudocomplemented} if for each $x\in P$ there exists a greatest element $y\in P$ such that $x\wedge y$ exists and is equal to $0$. This element $y$ is called the {\em pseudocomplement} of $x$ and it is denoted by $x^*$. Hence,
\[
x\wedge y=0\qquad\text{ if and only if }\qquad y\leq x^*.
\]
Thus $^*$ is a unary antitone operation on $P$ satisfying $x\le x^{**}$ and $x^{***}\approx x^*$. We will express the fact that $(P,\le,0,1)$ is pseudocomplemented by the notation
\[
\mathbf P=(P,\le,{}^*,0,1).
\]
In a pseudocomplemented poset $\mathbf P$, the orthogonality relation $\perp$ is defined by
\[
x\perp y\qquad\text{ if and only if }\qquad y\le x^*.
\]
Similarly as in Section~2 for meet-semilattices with $0$, also here we have that $\perp$ is symmetric and that $x\perp x$ if and only if $x=0$; of course, $x\perp 0$ for each $x\in P$. The definition of $A^{\perp}$ for $A\subseteq P$ is the same as in the previous sections and $A$ is called {\em closed} if $A^{\perp\perp}=A$.

For every subset $A$ of $P$, $A^*$ denotes the set $\{x^*\mid x\in A\}$.

If we consider pseudocomplemented lattices, we can simplify several of the aforementioned concepts and we use the so-called Glivenko Theorem to describe the structure of the lattice of all closed subsets of a complete pseudocomplemented lattice.

First recall the following concepts.

A {\em pseudocomplemented meet-semilattice} is an algebra $(S,\wedge,{}^*,0)$ of type $(2,1,0)$ such that $(S,\wedge,0)$ is a meet-semilattice with $0$ satisfying
\[
x\wedge y=0\text{ if and only if }y\le x^*.
\]
A {\em pseudocomplemented lattice} is an algebra $(L,\vee,\wedge,{}^*,0,1)$ of type $(2,2,1,0,0)$ such that $(L,\vee,\wedge,0,1)$ is a bounded lattice and $(L,\le,{}^*,0,1)$ is a pseudocomplemented poset.

We can describe infima of subsets of $P^*$ as well as subsets of the form $A^\perp$ with $A\subseteq P$. This result will be used in the proof of the next theorems.

\begin{lemma}\label{lem1}
	Let $(P,\le,{}^*,0,1)$ be a pseudocomplemented poset and $A\subseteq P$ and assume $\bigvee A$ to exist. Then the following holds:
	\begin{enumerate}[{\rm(i)}]
		\item $\bigwedge A^*$ exists and $\bigwedge A^*=(\bigvee A)^*$,
		\item $A^\perp=(\bigvee A)^\perp$.
	\end{enumerate}
\end{lemma}

\begin{proof}
	\
	\begin{enumerate}[(i)]
		\item For every $x\in A$ we have $x\le\bigvee A$ and hence $(\bigvee A)^*\le x^*$. This shows that $(\bigvee A)^*$ is a lower bound of $A^*$. Now let $a$ be a lower bound of $A^*$. Then $a\le x^*$ for all $x\in A$ and hence $x\le a^*$ for all $x\in A$. This shows that $a^*$ is an upper bound of $A$. Therefore $\bigvee A\le a^*$ which implies $a\le(\bigvee A)^*$. This shows that $\bigwedge A^*$ exists and $\bigwedge A^*=(\bigvee A)^*$.
		\item We have
		\begin{align*}
			A^\perp & =\bigcap_{a\in A}a^\perp=\bigcap_{a\in A}\{x\in P\mid x\le a^*\}=\{x\in P\mid x\le a^*\text{ for all }a\in A\}= \\
			& =\{x\in P\mid x\le\bigwedge A^*\}=\{x\in P\mid x\le\left(\bigvee A\right)^*\}=\left(\bigvee A\right)^\perp.
		\end{align*}
	\end{enumerate}	
\end{proof}

The following result follows from Lemma~\ref{lem1}. It shows that in case of a complete pseudocomplemented lattice the set $\Cl(\mathbf L)$ can be described in a simple way.

\begin{corollary}\label{cor1}
	Let $\mathbf L=(L,\vee,\wedge,{}^*,0,1)$ be a complete pseudocomplemented lattice. Then $\Cl(\mathbf L)=\{x^\perp\mid x\in L\}$.
\end{corollary}

\begin{proof}
	According to (ii) of Lemma~\ref{lem1} we have
	\[
	\Cl(\mathbf L)=\{A^\perp\mid A\subseteq L\}=\{\left(\bigvee A\right)^\perp\mid A\subseteq L\}=\{x^\perp\mid x\in L\}
	\]
	since every element $x$ of $L$ can be written in the form $\bigvee\{x\}$.
\end{proof}

The next theorem is the well-known Glivenko Theorem, see \cite{Gl}. Our version is taken from \cite{Gr}.

\begin{theorem}\label{th3}
	If $(S,\wedge,{}^*,0)$ is a pseudocomplemented meet-semilattice then
	\[
	(S^*,\sqcup,\wedge,{}^*,0,1)
	\]
	with
	\[
	x\sqcup y:=(x^*\wedge y^*)^*
	\]
	for all $x,y\in S^*$ is a Boolean algebra.
\end{theorem}

\begin{remark}
	If $(L,\vee,\wedge,{}^*,0,1)$ is a pseudocomplemented lattice then we have $x\sqcup y=(x\vee y)^{**}$ for all $x,y\in L$ according to Lemma~\ref{lem1}.
\end{remark}

The following lemma establishes a certain relationship between pseudocomplementation and orthogonality. This relationship will be used in the proof of next theorem.

\begin{lemma}\label{lem2}
	Let $(P,\le,{}^*,0,1)$ be a pseudocomplemented poset and $a\in P$. Then $(a^*)^\perp=a^{\perp\perp}$.
\end{lemma}

\begin{proof}
	For $b\in P$ the following are equivalent: $b\in(a^*)^\perp$; $b\le a^{**}$; $a^*\le b^*$; $x\le b^*$ for all $x\le a^*$; $b\perp x$ for all $x\in a^\perp$; $b\in a^{\perp\perp}$.
\end{proof}

In the next theorem we describe the structure of the lattice of closed subsets of a given complete pseudocomplemented lattice. Contrary to Theorem~\ref{th1}, we need not assume that $\mathbf L$ is atomic.

\begin{theorem}\label{th2}
	Let $\mathbf L=(L,\vee,\wedge,{}^*,0,1)$ be a complete pseudocomplemented lattice. Then $\BCl(\mathbf L)$ is a complete Boolean algebra.
\end{theorem}

\begin{proof}
	According to Lemma~\ref{lem3}, $\big(\Cl(\mathbf L),\vee,\cap,{}^\perp,\{0\},L\big)$ is a complete ortholattice with smallest element $\{0\}$ and greatest element $L$. According to Corollary~\ref{cor1} we have $\Cl(\mathbf L)=\{x^\perp\mid x\in L\}$. Let $a,b\in L$. Since $x^\perp=(x^*]$ for all $x\in L$, we have $a^*\le b^*$ if and only if $a^\perp\subseteq b^\perp$. This implies that $a^*=b^*$ if and only if $a^\perp=b^\perp$. We conclude that the mapping $x^*\mapsto x^\perp$ is a well-defined isomorphism from $(L^*,\le)$ to $\BCl(\mathbf L)$. Since the first poset is a Boolean algebra according to Theorem~\ref{th3}, the second one is a Boolean algebra, too. The unary operation of the first Boolean algebra maps $x^*$ onto $(x^*)^*$. Hence the unary operation of the second Boolean algebra maps $x^\perp$ onto $(x^*)^\perp$ which equals $(x^\perp)^\perp$ according to Lemma~\ref{lem2}. This means that $^\perp$ is the unary operation of the second Boolean algebra.
\end{proof}

The following example shows that $\BCl(\mathbf S)$ may be a Boolean algebra also in the case when $\mathbf S$ is not atomic.

\begin{example}
	If $L$ denotes the set of all open subsets of $\mathbb R$ and $A^*:=\bigcup\limits_{B\in L,B\subseteq\mathbb R\setminus A}B$ for all $A\in L$ then $\mathbf L:=(L,\cup,\cap,{}^*,\emptyset,\mathbb R)$ is a non-atomic complete pseudocomplemented lattice and $\BCl(\mathbf L)$ is a Boolean algebra according to Theorem~\ref{th2}.
\end{example}

The next example shows that $\BCl(\mathbf L)$ may be a Boolean algebra also in the case when $\mathbf L$ is not complete.

Let $\mathbb N$ denote the set of all positive integers.

\begin{example}
	The algebra
	\[
	\mathbf L:=(\{A\subseteq\mathbb N\mid A\text{ is finite or }\mathbb N\setminus A\text{ is finite}\},\cup,\cap,A\mapsto\mathbb N\setminus A,\emptyset,\mathbb N)
	\]
	is a non-complete atomic Boolean algebra and hence $\BCl(\mathbf L)\cong\mathbf2^{\{\{x\}\mid x\in\mathbb N\}}\cong\mathbf2^N$ according to Theorem~\ref{th1}.
\end{example}

As shown above, for a finite pseudocomplemented lattice $\mathbf L$ the lattice $\BCl(\mathbf L)$ is a Boolean algebra. However, in general, when orthogonality on a given set is not defined by pseudocomplementation or by disjointness, see e.g.\ \cite J, then this lattice need not share such a strong property. In fact, it was shown by J.~C.~Dacey, Jr.\ \cite D, by G.~Kalmbach \cite K and, alternatively, by the third author \cite L under what conditions the closed subsets form an orthomodular lattice.

The lattices of closed subsets for two pseudocomplemented lattices are shown in the following example.

\begin{example}
	$\text{}$ \\  	
	\begin{itemize}
		\item[{\rm(a)}] Consider the non-distributive pseudocomplemented lattice $\mathbf L_2=(L,\vee,\wedge,{}^*,0,1)$ depicted in Fig.~4:
		\begin{center}
			\setlength{\unitlength}{1.3mm}
			\begin{picture}(45,45)
				\put(10,25){\circle*{1.5}}
				\put(25,10){\circle*{1.5}}
				\put(25,40){\circle*{1.5}}
				\put(35,20){\circle*{1.5}}
				\put(35,30){\circle*{1.5}}
				\put(35,20){\line(0,1){10}}
				\put(25,10){\line(-1,1){15}}
				\put(25,10){\line(1,1){10}}
				\put(25,40){\line(-1,-1){15}}
				\put(25,40){\line(1,-1){10}}
				\put(24.2,6){$0$}
				\put(6,24){$b$}
				\put(37,19){$a$}
				\put(37,29){$c$}
				\put(24.2,42){$1$}
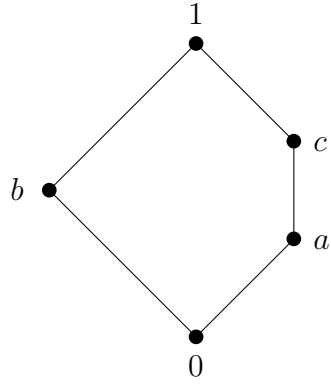
				\put(-14,0){{\rm Fig.~4: Non-distributive pseudocomplemented lattice $\mathbf L_2$}}
			\end{picture}
		\end{center}
		We have
		\[
		\begin{array}{l|ccccc}
			x      & 0 & a & b & c & 1 \\
			\hline
			x^*    & 1 & b & c & b & 0
		\end{array}
		\]
		The Boolean algebra $\BCl(\mathbf L_2)$ is visualized in Fig.~5:
		\begin{center}
			\setlength{\unitlength}{1.3mm}
			\begin{picture}(50,45)
				\put(10,25){\circle*{1.5}}
				\put(25,10){\circle*{1.5}}
				\put(25,40){\circle*{1.5}}
				\put(40,25){\circle*{1.5}}
				\put(25,10){\line(-1,1){15}}
				\put(25,10){\line(1,1){15}}
				\put(25,40){\line(-1,-1){15}}
				\put(25,40){\line(1,-1){15}}
				\put(22.7,6){$\{0\}$}
				\put(4,24){$(b]$}
				\put(42,24){$(c]$}
				\put(24.2,42){$L$}
				\put(3.5,0){{\rm Fig.~5: Boolean algebra $\BCl(\mathbf L_2)$}}
			\end{picture}
		\end{center}
		\vspace*{1mm} 
		\item[{\rm(b)}] Consider the pseudocomplemented lattice $\mathbf L_3=(L,\vee,\wedge,{}^*,0,1)$ depicted in Fig.~6:
		\begin{center}
			\setlength{\unitlength}{1.3mm}
			\begin{picture}(50,75)
				\put(10,25){\circle*{1.5}}
				\put(10,40){\circle*{1.5}}
				\put(25,10){\circle*{1.5}}
				\put(25,25){\circle*{1.5}}
				\put(25,40){\circle*{1.5}}
				\put(25,55){\circle*{1.5}}
				\put(40,25){\circle*{1.5}}
				\put(40,40){\circle*{1.5}}
				\put(25,70){\circle*{1.5}}
				\put(10,25){\line(0,1){15}}
				\put(40,25){\line(0,1){15}}
				\put(25,10){\line(0,1){15}}
				\put(25,40){\line(0,1){30}}
				\put(25,10){\line(-1,1){15}}
				\put(25,10){\line(1,1){15}}
				\put(25,25){\line(-1,1){15}}
				\put(25,25){\line(1,1){15}}
				\put(25,40){\line(-1,-1){15}}
				\put(25,40){\line(1,-1){15}}
				\put(25,55){\line(-1,-1){15}}
				\put(25,55){\line(1,-1){15}}
				\put(24.2,6){$0$}
				\put(6,24){$a$}
				\put(27,24){$b$}
				\put(42,24){$c$}
				\put(42,39){$f$}
				\put(6,39){$d$}
				\put(27,39){$e$}
				\put(27,55){$g$}
				\put(24.2,72){$1$}
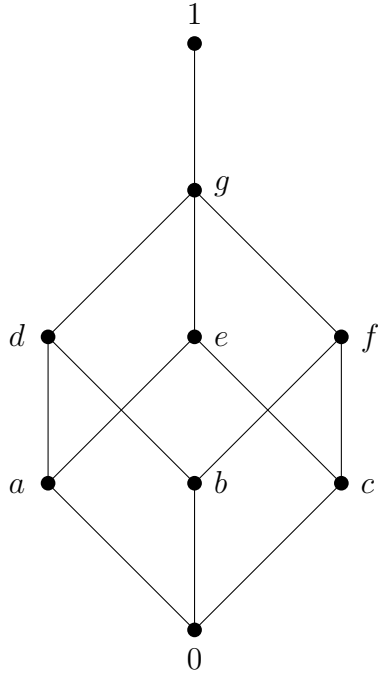
				\put(-2,0){{\rm Fig.~6: Pseudocomplemented lattice $\mathbf L_3$}}
			\end{picture}
		\end{center}
		We have
		\[
		\begin{array}{l|ccccccccc}
			x & 0 & a & b & c & d & e & f & g & 1 \\
			\hline
			x^* & 1 & f & e & d & c & b & a & 0 & 0
		\end{array} 
		\]
		The Boolean algebra $\BCl(\mathbf L_3)$ is visualized in Fig.~7:
		\begin{center}
			\setlength{\unitlength}{1.3mm}
			\begin{picture}(50,60)
				\put(10,25){\circle*{1.5}}
				\put(10,40){\circle*{1.5}}
				\put(25,10){\circle*{1.5}}
				\put(25,25){\circle*{1.5}}
				\put(25,40){\circle*{1.5}}
				\put(25,55){\circle*{1.5}}
				\put(40,25){\circle*{1.5}}
				\put(40,40){\circle*{1.5}}
				\put(10,25){\line(0,1){15}}
				\put(40,25){\line(0,1){15}}
				\put(25,10){\line(0,1){15}}
				\put(25,40){\line(0,1){15}}
				\put(25,10){\line(-1,1){15}}
				\put(25,10){\line(1,1){15}}
				\put(25,25){\line(-1,1){15}}
				\put(25,25){\line(1,1){15}}
				\put(25,40){\line(-1,-1){15}}
				\put(25,40){\line(1,-1){15}}
				\put(25,55){\line(-1,-1){15}}
				\put(25,55){\line(1,-1){15}}
				\put(22.7,6){$\{0\}$}
				\put(4,24){$(a]$}
				\put(27,24){$(b]$}
				\put(42,24){$(c]$}
				\put(42,39){$(f]$} 
				\put(4,39){$(d]$}
				\put(27,39){$(e]$}
				\put(23.8,57){$L$}
				\put(3.5,0){{\rm \rm Fig.~7: Boolean algebra $\BCl(\mathbf L_3)$}}
			\end{picture}
		\end{center}
	\end{itemize}
\end{example}

Analogous to Lemma~\ref{lem3} we have

\begin{theorem}\label{th4}
	If $\mathbf P=(P,\le,{}^*,0,1)$ be a pseudocomplemented poset then
	\[
	\BCl(\mathbf P):=\big(\Cl(\mathbf P),\vee,\cap,{}^\perp,\{0\},P\big)
	\]
	is a complete ortholattice with
	\begin{align*}
		\bigvee_{i\in I}A_i & =\left(\bigcup_{i\in I}A_i\right)^{\perp\perp}, \\
		\bigwedge_{i\in I}A_i & =\bigcap_{i\in I}A_i
	\end{align*}
	for all families $A_i,i\in I,$ of elements of $\Cl(\mathbf P)$.
\end{theorem}

\begin{proof}
	It is easy to see that
	\[ 
	\bigcap_{i\in I}A_i^\perp=\left(\bigcup_{i\in I}A_i\right)^\perp
	\]
	for all families $A_i,i\in I,$ of elements of $\Cl(\mathbf P)$ and hence $\big(\Cl(\mathbf P),\subseteq\big)$ is a complete lattice whose meet-operation coincides with set-theoretical intersection. The unary operation $^\perp$ on $\Cl(\mathbf P)$ is an antitone involution and because of $A\cap A^\perp=\{0\}$ for all $A\subseteq P$ also a complementation.
\end{proof}

\section{Pseudocomplemented posets}

In this section we investigate pseudocomplemented posets whose complete ortholattice of closed subsets forms a Boolean algebra. We start with the following example.

\begin{example}\label{ex2}
	Consider the poset $\mathbf P$ visualized in Fig.~8:
	\begin{center}
		\setlength{\unitlength}{1.3mm}
		\begin{picture}(45,45)
			\put(15,20){\circle*{1.5}}
			\put(15,30){\circle*{1.5}}
			\put(5,30){\circle*{1.5}}
			\put(45,30){\circle*{1.5}}
			\put(25,10){\circle*{1.5}}
			\put(25,40){\circle*{1.5}}
			\put(35,20){\circle*{1.5}}
			\put(35,30){\circle*{1.5}}
			\put(35,20){\line(0,1){10}} 
			\put(15,20){\line(0,1){10}}
			\put(15,20){\line(2,1){20}}
			\put(35,20){\line(-2,1){20}}
			\put(25,10){\line(-1,1){20}}
			\put(25,10){\line(1,1){20}}
			\put(25,40){\line(-1,-1){10}}
			\put(25,40){\line(1,-1){10}}
			\put(25,40){\line(2,-1){20}}
			\put(25,40){\line(-2,-1){20}}
			\put(24.2,6){$0$}
			\put(11,19){$a$}
			\put(12,29){$d$}
			\put(37,19){$b$}
			\put(36.5,29){$e$}
			\put(1.5,29){$c$}
			\put(46.5,29){$f$}
			\put(24.2,42){$1$}
			\put(14,0){{\rm Fig.~8: Poset $\mathbf P$}}
		\end{picture}
	\end{center}
	Evidently, $\mathbf P$ is not a lattice. We have
	\[
	\begin{array}{l|cccccccc}
		x   & 0 & a & b & c & d & e & f & 1 \\
		\hline
		x^* & 1 & f & c & f & 0 & 0 & c & 0
	\end{array}
	\]
and
\[
0^\perp=P, a^\perp=c^\perp=\{0,b,f\}=(f], b^\perp=f^\perp=\{0,a,c\}=(c], d^\perp=e^\perp=1^\perp=\{0\}.
\]
These are all subsets closed with respect to orthogonality. They form a Boolean algebra with respect to set inclusion, see Fig.~9.
	\begin{center}
		\setlength{\unitlength}{1.3mm}
		\begin{picture}(45,45)
			\put(10,25){\circle*{1.5}}
			\put(25,10){\circle*{1.5}}
			\put(25,40){\circle*{1.5}}
			\put(40,25){\circle*{1.5}}
			\put(25,10){\line(-1,1){15}}
			\put(25,10){\line(1,1){15}}
			\put(25,40){\line(-1,-1){15}}
			\put(25,40){\line(1,-1){15}}
			\put(23,6){$\{0\}$}
			\put(4,24){$(c]$}
			\put(42,24){$(f]$}
			\put(24.2,42){$P$}
			\put(4,0){{\rm \rm Fig.~9\: Boolean algebra $\BCl(\mathbf P)$}}
		\end{picture}
	\end{center}
\end{example}

It is a question if the situation described in Example \ref{ex2} can be generalized, i.e.\ if for each pseudocomplemented poset $\mathbf P$ the complete ortholattice $\BCl(\mathbf P)$ of its closed subsets forms a Boolean algebra.

For this, we firstly prove a result similar to the famous theorem of Glivenko but formulated only for those pseudocomplemented posets $(P,\le;{}^*,0,1)$ where $(P^*,\le)$ is a meet-semilattice.

\begin{theorem}\label{th5}
	Let $\mathbf P=(P,\le,{}^*,0,1)$ be a pseudocomplemented poset where $(P^*,\le)$ is a meet-semilattice $(P^*,\wedge)$ and put
	\[
	x\vee y:=(x^*\wedge y^*)^*
	\]
	for all $x,y\in P^*$. Then $(P^*,\vee,\wedge,{}^*,0,1)$ is a Boolean algebra.
\end{theorem}

\begin{proof}\
	First observe that $^*$ is an antitone involution on $(P^*,\le)$. Hence $(P^*,\vee,\wedge)$ is a lattice and De Morgan's laws hold. Since $\mathbf P$ satisfies the identity $x\wedge x^*\approx0$, the operation $^*$ is a complementation. Now let $a,b,c\in P^*$ and put $d:=a\vee(b\wedge c)$. Then $a\wedge b\le d$ and hence $a\wedge b\wedge d^*=0$ whence
	\begin{enumerate}
		\item[(1)] $b\wedge d^*\le a^*$.
	\end{enumerate}
	Since $b\wedge c\le d$ we obtain $b\wedge c\wedge d^*=0$ which implies $b\wedge d^*\le c^*$. This together with (1) yields $b\wedge d^*\le a^*\wedge c^*$. From this we get $b\wedge(a\vee c)\wedge d^*=0$ and hence
	\begin{enumerate}
		\item[(2)] $(a\vee c)\wedge d^*\le b^*$.
	\end{enumerate}
	Now we have $a\le d$ which implies $a\wedge d^*=0$. Therefore $a\wedge(a\vee c)\wedge d^*=0$ from which we get $(a\vee c)\wedge d^*\le a^*$. This together with (2) yields $(a\vee c)\wedge d^*\le a^*\wedge b^*$. From this we derive $(a\vee c)\wedge(a\vee b)\wedge d^*=0$ and hence $(a\vee c)\wedge(a\vee b)\le d$ showing $(P^*,\vee,\wedge)$ to be distributive.
\end{proof}

The question when for a subset $A$ of $P$ there exists some $b\in P$ satisfying $b^\perp=A^\perp$ is solved in the following lemma. However, we must notice that for a pseudocomplemented poset $(P,\le,{}^*,0,1)$ there may happen the following rather pathological situation. Namely, if $(P^*,\le)$ forms a complete lattice, $A\subseteq P^*$ and $a$ denotes the infimum of $A$ in $(P^*,\le)$ then there may exist some lower bound $b$ of $A$ in $(P,\le)$ satisfying $b\not\le a$. In such a case, the infimum of $A$ in the whole poset $(P,\le)$ either does not exists or it differs from $a$. In order to avoid such a pathological situation, we must require one more assumption in the next formulations.

\begin{lemma}
Let $(P,\le,{}^*,0,1)$ be a pseudocomplemented poset such that $(P^*,\le)$ is a complete lattice having the property that for each subset $B$ of $P^*$ the infimum of $B$ in $(P,\le)$ exists and coincides with the infimum of $B$ in $(P^*,\le)$ and let $A$ be a subset of $P$. Then there exists some $b\in P$ with $b^\perp=A^\perp$.
\end{lemma}

\begin{proof}
Since $\bigwedge A^*\in P^*$ there exists some $b\in P$ with $b^*=\bigwedge A^*$. Now we have
\begin{align*}
	A^\perp & =\bigcap_{a\in A}a^\perp=\bigcap_{a\in A}\{x\in P\mid x\le a^*\}=\{x\in P\mid x\le a^*\text{ for all }a\in A\}= \\
	& =\{x\in P\mid x\le\bigwedge A^*\}=\{x\in P\mid x\le b^*\}=b^\perp.
\end{align*}
\end{proof}

\begin{corollary}\label{cor2}
If $(P,\le,{}^*,0,1)$ be a pseudocomplemented poset such that $(P^*,\le)$ is a complete lattice having the property that for each subset $B$ of $P^*$ the infimum of $B$ in $(P,\le)$ exists and coincides with the infimum of $B$ in $(P^*,\le)$ then $\Cl(\mathbf P)=\{x^\perp\mid x\in P\}$.
\end{corollary}

Using several of the results obtained so far we can show that if $(P^*,\le)$ is a complete lattice having the property that for each subset $B$ of $P^*$ the infimum of $B$ in $(P,\le)$ exists and coincides with the infimum of $B$ in $(P^*,\le)$, the complete ortholattice of closed subsets of $\mathbf P$ forms again a Boolean algebra.

\begin{theorem}
Let $(P,\le,{}^*,0,1)$ be a pseudocomplemented poset such that $(P^*,\le)$ is a complete lattice having the property that for each subset $B$ of $P^*$ the infimum of $B$ in $(P,\le)$ exists and coincides with the infimum of $B$ in $(P^*,\le)$. Then $\BCl(\mathbf P)$ is a complete Boolean algebra.
\end{theorem}

\begin{proof}
	According to Theorem~\ref{th4}, $\big(\Cl(\mathbf P),\vee,\cap,{}^\perp,\{0\},P\big)$ is a complete ortholattice with smallest element $\{0\}$ and greatest element $P$. According to 
	Corollary~\ref{cor2} we have $\Cl(\mathbf P)=\{x^\perp\mid x\in P\}$. Let $a,b\in P$. Since $x^\perp=(x^*]$ for all $x\in P$, we have $a^*\le b^*$ if and only if $a^\perp\subseteq b^\perp$. This implies that $a^*=b^*$ if and only if $a^\perp=b^\perp$. We conclude that the mapping $x^*\mapsto x^\perp$ is a well-defined isomorphism from $(P^*,\le)$ to $\BCl(\mathbf P)$. Since the first poset is a Boolean algebra according to Theorem~\ref{th5}, the second one is a Boolean algebra, too. The unary operation of the first Boolean algebra maps $x^*$ onto $(x^*)^*$. Hence the unary operation of the second Boolean algebra maps $x^\perp$ onto $(x^*)^\perp$ which equals $(x^\perp)^\perp$ according to Lemma~\ref{lem2}. This means that $^\perp$ is the unary operation of the second Boolean algebra.
\end{proof}

In previous assertions we assume that $x^*\wedge y^*$ exists in $P^*$. The question arises if this condition can be omitted, i.e.\ if there exists some finite pseudocomplemented poset $\mathbf P$ with the property that its subset $P^*$ does not form a meet-semilattice. We now present such an example.

\begin{example}\label{ex3}
	Fig.~10 shows a pseudocomplemented non-lattice poset $\mathbf P=(P,\le,{}^*,0,1)$ where e.g.\ the elements $e$ and $f$ have no infimum in $(P^*,\le)$ and the elements $a$ and $b$ have no supremum in $(P^*,\le)$. Thus $(P^*,\le)=(P,\le)$ is not a lattice.
\begin{center}
\setlength{\unitlength}{1.3mm}
\begin{picture}(105,60)
\put(25,25){\circle*{1.5}}
\put(25,40){\circle*{1.5}}
\put(40,25){\circle*{1.5}}
\put(40,40){\circle*{1.5}}
\put(47.5,10){\circle*{1.5}}
\put(47.5,55){\circle*{1.5}}
\put(55,25){\circle*{1.5}}
\put(55,40){\circle*{1.5}}
\put(70,25){\circle*{1.5}}
\put(70,40){\circle*{1.5}}
\put(47.5,10){\line(-3,2){22.5}}
\put(47.5,10){\line(-1,2){7.5}}
\put(47.5,10){\line(1,2){7.5}}
\put(47.5,10){\line(3,2){22.5}}
\put(25,25){\line(0,1){15}}
\put(25,25){\line(1,1){15}}
\put(25,25){\line(2,1){30}}
\put(40,25){\line(-1,1){15}}
\put(40,25){\line(0,1){15}}
\put(40,25){\line(2,1){30}}
\put(55,25){\line(-2,1){30}}
\put(55,25){\line(0,1){15}}
\put(55,25){\line(1,1){15}}
\put(70,25){\line(-2,1){30}}
\put(70,25){\line(-1,1){15}}
\put(70,25){\line(0,1){15}}
\put(25,40){\line(3,2){22.5}}
\put(40,40){\line(1,2){7.5}}
\put(55,40){\line(-1,2){7.5}}
\put(70,40){\line(-3,2){22.5}}
\put(46.8,6){$0$}
\put(21,24){$a$}
\put(36,24){$b$}
\put(57,24){$c$}
\put(72,24){$d$}
\put(21,39){$e$}
\put(36,39){$f$}
\put(57,39){$g$}
\put(72,39){$h$}
\put(46.3,57){$1$}
\put(20.5,0){{\rm Fig.~10: Pseudocomplemented poset $\mathbf P$}}
\end{picture}\
\end{center}
We have
	\[
	\begin{array}{l|llllllllllll}
		x   & 0 & a & b & c & d & e & f & g & h & 1 \\
		\hline
		x^* & 1 & h & g & f & e & d & c & b & a & 0
	\end{array}
	\]
	and $\mathbf P$ satisfies the identity $x^{**}\approx x$.
\end{example}

\begin{example}
	Consider again the pseudocomplemented poset $\mathbf P$ from Example~\ref{ex3}. Although $(P^*,\le)$ is not a meet-semilattice, we can show rather surprisingly that the complete ortholattice of closed subsets of $\mathbf P$ forms a $16$-element Boolean algebra.
	
	It is easy to verify that for every $p\in P$, $(p]$ is a closed subset of $\mathbf P$. Besides of these there are exactly six closed subsets of $\mathbf P$, namely $\{0,a,b\}$, $\{0,a,c\}$, $\{0,a,d\}$, $\{0,b,c\}$, $\{0,b,d\}$ and $\{0,c,d\}$. The corresponding lattice $\BCl(\mathbf P)$ is visualized in Fig.~11.
\begin{center}
	\setlength{\unitlength}{1.3mm}
	\begin{picture}(105,75)
		\put(10,40){\circle*{1.5}}
		\put(25,25){\circle*{1.5}}
		\put(25,55){\circle*{1.5}}
		\put(40,25){\circle*{1.5}}
		\put(40,55){\circle*{1.5}}
		\put(47.5,10){\circle*{1.5}}
		\put(47.5,70){\circle*{1.5}}
		\put(55,25){\circle*{1.5}}
		\put(55,55){\circle*{1.5}}
		\put(70,25){\circle*{1.5}}
		\put(70,55){\circle*{1.5}}
		\put(85,40){\circle*{1.5}}
		\put(25,40){\circle*{1.5}}
		\put(40,40){\circle*{1.5}}
		\put(55,40){\circle*{1.5}}
		\put(70,40){\circle*{1.5}}
		\put(47.5,10){\line(-3,2){22.5}}
		\put(47.5,10){\line(-1,2){7.5}}
		\put(47.5,10){\line(1,2){7.5}}
		\put(47.5,10){\line(3,2){22.5}}
		\put(25,25){\line(-1,1){15}}
		\put(25,25){\line(1,1){30}}
		\put(40,25){\line(-2,1){30}}
		\put(40,25){\line(1,1){30}}
		\put(55,25){\line(2,1){30}}
		\put(70,25){\line(1,1){15}}
		\put(10,40){\line(1,1){15}}
		\put(10,40){\line(2,1){30}}
		\put(85,40){\line(-2,1){30}}
		\put(85,40){\line(-1,1){15}}
		\put(25,55){\line(3,2){22.5}}
		\put(40,55){\line(1,2){7.5}}
		\put(55,55){\line(-1,2){7.5}}
		\put(70,55){\line(-3,2){22.5}}
		\put(25,25){\line(0,1){30}}
		\put(40,25){\line(2,1){30}}
		\put(55,25){\line(-2,1){30}}
		\put(55,25){\line(0,1){15}}
		\put(70,25){\line(-2,1){30}}
		\put(70,25){\line(0,1){30}}
		\put(25,40){\line(2,1){30}}
		\put(40,40){\line(0,1){15}}
		\put(55,40){\line(-2,1){30}}
		\put(70,40){\line(-2,1){30}}
		\put(45.2,6){$\{0\}$}
		\put(19,24){$(a]$}
		\put(34,24){$(b]$}
		\put(57,24){$(c]$}
		\put(72,24){$(d]$}
		\put(-2.5,39){$\{0,a,b\}$}
		\put(13,39){$\{0,a,c\}$}
		\put(28,39){$\{0,a,d\}$}
		\put(56,39){$\{0,b,c\}$}
		\put(71,39){$\{0,b,d\}$}
		\put(87,39){$\{0,c,d\}$}
		\put(19,54){$(e]$}
		\put(34,54){$(f]$}
		\put(57,54){$(g]$}
		\put(72,54){$(h]$}
		\put(46.3,72){$P$}
		\put(25.5,0){{\rm Fig.~11: Boolean algebra $\BCl(\mathbf P)$}}
	\end{picture}
\end{center}
\end{example}

Finally, we show that the poset depicted in Fig.~12 is just the forbidden configuration for a pseudocomplemented poset $(Q,\le,{}^*,0,1)$ satisfying the ascending chain condition in order to have $(Q^*,\le)$ as a meet-semilattice.

\begin{theorem}
	Let $(Q,\le,{}^*,0,1)$ be a pseudocomplemented poset satisfying the ascending chain condition. Then $(Q^*,\le)$ is a meet-semilattice if and only if $(Q,\le)$ does not contain a configuration of the form visualized in Fig.~12 where $f^*$ and $g^*$ are maximal lower bounds of $a^*$ and $d^*$ in $(Q^*,\le)$, $a^{**}<d^*$ if and only if $d^{**}<a^*$, and $g^*<f^{**}$ if and only if $f^*<g^{**}$.
\begin{center}
	\setlength{\unitlength}{1.3mm}
	\begin{picture}(105,60)
		\put(25,25){\circle*{1.5}}
		\put(25,40){\circle*{1.5}}
		\put(40,25){\circle*{1.5}}
		\put(40,40){\circle*{1.5}}
		\put(47.5,10){\circle*{1.5}}
		\put(47.5,55){\circle*{1.5}}
		\put(55,25){\circle*{1.5}}
		\put(55,40){\circle*{1.5}}
		\put(70,25){\circle*{1.5}}
		\put(70,40){\circle*{1.5}}
		\multiput(25,25)(15,0){4}{\line(0,1){15}}
		\linethickness{.5mm}
		\put(47.5,10){\line(-3,2){22.5}}
		\put(47.5,10){\line(-1,2){7.5}}
		\put(47.5,10){\line(1,2){7.5}}
		\put(47.5,10){\line(3,2){22.5}}
		\put(25,25){\line(1,1){15}}
		\put(25,25){\line(2,1){30}}
		\put(40,25){\line(-1,1){15}}
		\put(40,25){\line(2,1){30}}
		\put(55,25){\line(-2,1){30}}
		\put(55,25){\line(1,1){15}}
		\put(70,25){\line(-2,1){30}}
		\put(70,25){\line(-1,1){15}}
		\put(25,40){\line(3,2){22.5}}
		\put(40,40){\line(1,2){7.5}}
		\put(55,40){\line(-1,2){7.5}}
		\put(70,40){\line(-3,2){22.5}}
		\put(46.8,6){$0$}
		\put(20,24){$a^{**}$}
		\put(35,24){$g^*$}
		\put(57,24){$f^*$}
		\put(72,24){$d^{**}$}
		\put(20,39){$d^*$}
		\put(35,39){$f^{**}$}
		\put(57,39){$g^{**}$}
		\put(72,39){$a^*$}
		\put(46.3,57){$1$}
		\put(24.5,0){{\rm Fig.~12: Forbidden configuration}}
	\end{picture}\
\end{center}
\end{theorem}

\begin{proof}
	If $(Q^*,\le)$ contains a configuration of the form depicted in Fig.~12 satisfying the above mentioned properties then $(Q^*,\le)$ is not a meet-semilattice since $a^*$ and $d^*$ have no infimum in $(Q^*,\le)$. Conversely, assume $(Q^*,\le)$ not to be a meet-semilattice. Then there exist elements $a^*$ and $d^*$ of $Q^*$ having no infimum in $(Q^*,\le)$. Let $f^*$ and $g^*$ be different maximal lower bounds of $a^*$ and $d^*$ in $(Q^*,\le)$. It is clear that $a^*\parallel d^*$ and $f^*\parallel g^*$. From this we conclude that $a^*$, $d^*$, $f^*$ and $g^*$ are different from $0$ and $1$ and hence the same is true for $a^{**}$, $d^{**}$, $f^{**}$ and $g^{**}$. Therefore $a^*\parallel a^{**}$, $d^*\parallel d^{**}$, $f^*\parallel f^{**}$ and $g^*\parallel g^{**}$. We have $f^*\le a^*$, $f^*\le d^*$, $g^*\le a^*$ and $g^*\le d^*$ and hence $a^{**}\le f^{**}$, $a^{**}\le g^{**}$, $d^{**}\le f^{**}$ and $d^{**}\le g^{**}$. It is clear that $a^{**}\le d^*$ if and only if $d^{**}\le a^*$ and that $g^*\le f^{**}$ if and only if $f^*\le g^{**}$. Now $d^*\le a^{**}$ would imply $f^*\le d^*\le a^{**}\le f^{**}$ contradicting $f^*\parallel f^{**}$. This shows $d^*\not\le a^{**}$. Moreover, $f^{**}\le g^*$ would imply $a^{**}\le f^{**}\le g^*\le a^*$ contradicting $a^*\parallel a^{**}$. Therefore we have $f^{**}\not\le g^*$. Next we show that the eight elements $a^{**}$, $g^*$, $f^*$, $d^{**}$, $d^*$, $f^{**}$, $g^{**}$ and $a^*$ are pairwise different. We have $f^*\ne g^*$ and because of $a^*\ne d^*$ also $a^{**}\ne d^{**}$. Now $a^{**}\in\{f^*,g^*\}$ would imply $a^{**}\le a^*$ contradicting $a^*\parallel a^{**}$. Similarly, $d^{**}\in\{f^*,g^*\}$ would imply $d^{**}\le d^*$ contradicting $d^*\parallel d^{**}$. Hence $a^{**}$, $g^*$, $f^*$ and $d^{**}$ and therefore also $d^*$, $f^{**}$, $g^{**}$ and $a^*$ are pairwise different. Now $a^{**}=d^*$ would imply $g^*\le a^*$ and $g^*\le a^{**}$, a contradiction. $a^{**}=f^{**}$ would imply $a^*=f^*\le d^*$ contradicting $a^*\parallel d^*$. $a^{**}=g^{**}$ would imply $a^*=g^*\le d^*$, contradicting $a^*\parallel d^*$. Finally, $a^{**}=a^*$ would contradict $a^*\parallel a^{**}$. This shows $a^{**}\not\in\{d^*,f^{**},g^{**},a^*\}$. Now $g^*=d^*$ would imply $d^*=g^*\le a^*$ contradicting $a^*\parallel d^*$. $g^*=f^{**}$ would imply $a^{**}\le f^{**}=g^*\le a^*$ contradicting $a^*\parallel a^{**}$. $g^*=g^{**}$ would contradict $g^*\parallel g^{**}$. Finally, $g^*=a^*$ would imply $a^*=g^*\le d^*$ contradicting $a^*\parallel d^*$. This shows $g^*\not\in\{d^*,f^{**},g^{**},a^*\}$. From symmetry reasons we obtain $\{f^*,d^{**}\}\cap\{d^*,f^{**},g^{**},a^*\}=\emptyset$. Thus we have shown that the eight elements $a^{**}$, $g^*$, $f^*$, $d^{**}$, $d^*$, $f^{**}$, $g^{**}$ and $a^*$ are pairwise different completing the proof of the theorem.
\end{proof}

Observe that the configuration visualized in Fig.~12 is included in Fig.~10.

\begin{corollary}
	If $(Q,\le,{}^*,0,1)$ is a pseudocomplemented poset satisfying the ascending chain condition and having the property that $(Q^*,\le)$ is not a meet-semilattice then $(Q^*,\le)$ contains a configuration of the form depicted in Fig.~13.
	\begin{center}
	\setlength{\unitlength}{1.3mm}
	\begin{picture}(45,45)
		\put(15,20){\circle*{1.5}}
		\put(15,30){\circle*{1.5}}
		\put(25,10){\circle*{1.5}}
		\put(25,40){\circle*{1.5}}
		\put(35,20){\circle*{1.5}}
		\put(35,30){\circle*{1.5}}
		\put(35,20){\line(0,1){10}}
		\put(15,20){\line(0,1){10}}
		\put(15,20){\line(2,1){20}}
		\put(35,20){\line(-2,1){20}}
		\put(25,10){\line(-1,1){10}}
		\put(25,10){\line(1,1){10}}
		\put(25,40){\line(-1,-1){10}}
		\put(25,40){\line(1,-1){10}}
		\put(24.2,6){$0$}
		\put(10.5,19){$g^*$}
		\put(10.5,29){$d^*$}
		\put(37,19){$f^*$}
		\put(37,29){$a^*$}
		\put(24.2,42){$1$}
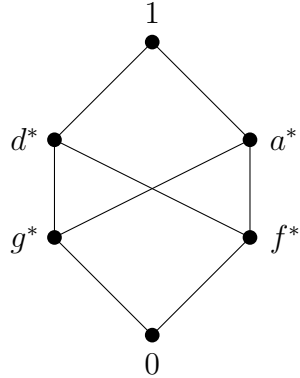
		\put(-3.5,0){{\rm Fig.~13: Another forbidden configuration}}
	\end{picture}
\end{center}
\end{corollary}

{\bf Data availability statement} No datasets were generated or analyzed during the current study.

%
%
%
%

Authors' addresses:

Ivan Chajda \\
Palack\'y University Olomouc \\
Faculty of Science \\
Department of Algebra and Geometry \\
17.\ listopadu 12 \\
771 46 Olomouc \\
Czech Republic \\
ivan.chajda@upol.cz

Miroslav Kola\v r\'ik \\
Palack\'y University Olomouc \\
Faculty of Science \\
Department of Computer Science \\
17.\ listopadu 12 \\
771 46 Olomouc \\
Czech Republic \\
miroslav.kolarik@upol.cz

Helmut L\"anger \\
TU Wien \\
Faculty of Mathematics and Geoinformation \\
Institute of Discrete Mathematics and Geometry \\
Wiedner Hauptstra\ss e 8-10 \\
1040 Vienna \\
Austria, and \\
Palack\'y University Olomouc \\
Faculty of Science \\
Department of Algebra and Geometry \\
17.\ listopadu 12 \\
771 46 Olomouc \\
Czech Republic \\
helmut.laenger@tuwien.ac.at


\begin{thebibliography}9
\bibitem D
J.~C.~Dacey, Jr, Orthomodular Spaces. PhD thesis University of Massachusetts Amherst (1968).
\bibitem{Gl}
V.~Glivenko, Sur quelques points de la logique de M.~Brouwer. Bulletin Acad.\ Bruxelles {\bf15} (1929), 183--188.
\bibitem{Gr}
G.~Gr\"atzer, Lattice Theory: Foundation. Birkh\"auser/Springer, Basel 2011. ISBN 978-3-0348-0017-4.
\bibitem J
G.~Jen\v ca, Orthogonality spaces associated with posets. Order {\bf40} (2023), 575--588.
\bibitem K
G.~Kalmbach, Orthomodular Lattices. Academic Press, London 1983. ISBN 0-12-394580-1.
\bibitem L
H.~L\"anger, Another characterization of the orthomodularity of the ortholattice of all polars. Tatra Mt.\ Math.\ Publ.\ {\bf3} (1993), 53--54.
\bibitem{LV}
B.~Lindenhovius and T.~Vetterlein, A characterization of orthomodular spaces by Sasaki maps. Internat.\ J.\ Theoret.\ Phys.\ {\bf62} (2023). 59 (14 pp.).
\bibitem{PV}
J.~Paseka and T.~Vetterlein, Linear orthosets and orthogeometries. Internat.\ J.\ Theoret.\ Phys.\ {\bf62} (2023), 55 (15 pp.).
\end{thebibliography}
\end{document}